\documentclass[reqno]{amsart}
\usepackage{amsmath,amsthm,amssymb}
\usepackage[utf8]{inputenc}
\usepackage[all]{xy}

\usepackage[colorlinks, citecolor=black, filecolor=black, linkcolor=black,urlcolor=black,]{hyperref} 

\newcommand{\kk}             {{\mathbb K}}
\newcommand{\tens}          {\otimes}
\newcommand{\into}           {\rightarrow}
\newcommand{\sumim}      {\sum\limits_{i=1}^m}
\newcommand{\fhi}             {\varphi}
\newcommand{\sig}            {\Sigma}
\newcommand{\lev}            {\left\langle}
\newcommand{\rev}            {\right\rangle}
\newcommand{\so}             {\Sigma\mbox{-ope\-ra\-tor}}
\newcommand{\sos}           {\Sigma\mbox{-ope\-ra\-tors}}

\newcommand{\xxx}              {X_1,\dots, X_n}
\newcommand{\eee}             {E_1,\dots, E_n}
\newcommand{\eel}              {E_1,\dots ,E_n,L}
\newcommand{\mmn}           {M_1,\dots ,M_n,N}
\newcommand{\xpx}             {X_1\times\dots\times X_n}
\newcommand{\epe}            {E_1\times\dots\times E_n}
\newcommand{\xxp}             {(x^1,\dots ,x^n)}
\newcommand{\zzp}              {(z^1,\dots ,z^n)}

\newcommand{\sxx}              {\sig_{X_1\dots X_n}}
\newcommand{\sxxp}            {\sig_{X_1 \dots X_n}^\pi}
\newcommand{\sxxb}            {\sig_{X_1 \dots X_n}^\beta}
\newcommand{\seee}           {\sig_{E_1\dots E_n}}

\newcommand{\xtx}      		   {X_1\tens\dots\tens X_n}
\newcommand{\ete}      		   {E_1\tens\dots\tens E_n}
\newcommand{\xtxty}    	  	   {X_1\tens\dots\tens X_n\tens Y}
\newcommand{\etetf}            {E_1\tens\dots \tens E_n\tens F}
\newcommand{\xtxtyd}   	   {X_1\tens\dots\tens X_n\tens Y^*}
\newcommand{\xtxp}            {X_1\widehat{\tens}_\pi\dots\widehat{\tens}_\pi X_n}
\newcommand{\xtxtyp}         {X_1\widehat{\tens}_\pi\dots\widehat{\tens}_\pi X_n\widehat{\tens}_\pi Y}
\newcommand{\pmq}            {p-q}
\newcommand{\xxt}              {x^1\tens\dots\tens  x^n}
\newcommand{\zzt}              {z^1\tens\dots\tens z^n}
\newcommand{\pmqy}          {(\pmq)\tens y}
\newcommand{\pimqi}          {p_i-q_i}
\newcommand{\ajmbj}          {a_j-b_j}
\newcommand{\aimbi}          {a_i-b_i}
\newcommand{\rju}               {\sumim(p_i-q_i)\tens y_i}

\newcommand{\Txxy}           {T:\xpx\into Y}
\newcommand{\tlin}             {\widetilde{T}}

\newcommand{\Lxx}             {\mathcal{L}\left(\xxx\right)}
\newcommand{\Lxxy}           {\mathcal{L}\left(\xxx,Y\right)}
\newcommand{\Gxxy}          {\Gamma(\xxx; Y)}
\newcommand{\Gxxyd}        {\Gamma(\xxx; Y^*)}

\newtheorem*{theokwapien}  {Kwapie\'n Type Characterization}
\newtheorem{definition}         {Definition}[section]
\newtheorem{proposition}      [definition]{Proposition}
\newtheorem{lemma}             [definition]{Lemma}
\newtheorem{corollary}          [definition]{Corollary}
\newtheorem{theorem}          [definition]{Theorem}

\begin{document}

\title{Multilinear Operators Factoring through Hilbert Spaces}
\author{Maite Fern\'{a}ndez-Unzueta, Samuel Garc\'{i}a-Hern\'{a}ndez}
\address{Centro de Investigaci\'{o}n en Matem\'{a}ticas (Cimat), A.P. 402 Guanajuato, Gto., M\'{e}xico}
\email{maite@cimat.mx; orcid:{0000-0002-8321-4877}} 
\email{samuelg@cimat.mx; orcid:{0000-0003-4562-299X}} 






\begin{abstract}
We characterize those bounded multilinear operators that factor through a Hilbert space in terms of its behavior in finite sequences. This extends a result, essentially due to S. Kwapie\'n, from the linear to the multilinear setting. We prove that Hilbert-Schmidt and Lipschitz $2$-summing multilinear operators naturally factor through a Hilbert space. It is also proved that the class $\Gamma$ of all multilinear operators that factor through a Hilbert space is a maximal multi-ideal; moreover, we give an explicit formulation of a finitely generated tensor norm $\gamma$ which is in duality with $\Gamma$.
\end{abstract}
\maketitle


\section{Introduction}

The fact  that a bounded linear operator between Banach spaces factors through a Hilbert space is,  a priori, a fairly  abstract property. It is possible,  however,  to describe it  in terms of the behaviour of the operator in a special type of finite sequences of the domain.  Such a local expression of the property makes it possible to relate it with other fundamental notions of the Geometry of Banach spaces.  This  is the case, for example,  of  Kwapie\'n's characterization of the  Banach spaces that are  isomorphic to a Hilbert space as those having type  2  and cotype 2 \cite{kwapien72a}. Regarding   the factorization of  linear operators  through a Hilbert space, we refer the reader to  the original  papers  \cite{Groth, LP},   and  to the corresponding chapters within the treatises \cite{DJT, pisier86, TJ}.

The problem of factoring an operator through a Hilbert space has also been studied  for other mappings than  linear operators. This is the case of Lispchitz mapings between metric spaces and completely bounded operators between operator spaces developed in \cite{chavez14} and \cite{pisier96}, respectively. This problem certainly makes sense for multilinear mappings.  Compact \cite{krikorian72}, nuclear \cite{alencar85}, $p$-summing \cite{angulo17, dimant03, pietsch83}, integral \cite{villanueva03} and other classes of linear operators have been extended to the multilinear setting. Despite this, the case of factoring a multilinear operator through a Hilbert space, as far as we know,  has not been studied nor even defined. In this paper we provide a solution to this problem. Now, we briefly describe our results:

We say that $\Txxy$ factors through a  Hilbert space if there exist a Hilbert space $H$, a subset $M$ of $H$, a bounded multilinear operator $A:\xpx\into H$ and a Lipschitz function $B:M\into Y$ such that $A(\xpx)\subset M$ and
\begin{equation}\label{multilineardiagram}
\begin{array}{c}
\xymatrix{
\xpx\ar[dr]_{A}\ar[rr]^-T   &                              &Y\\
                                              &M\ar[ur]_B\ar[d]_i  &\\
                                              &H                           &
}
\end{array}
\end{equation}
commutes, that is,  $T=B\circ A$. We define $\Gamma(T)=\inf \|A\|\,  Lip(B)$, where the infimum is taken over all possible factorizations as in \eqref{multilineardiagram}.

Our main result,  Theorem \ref{kwapien}, says the following:  If $\pi$ denotes the projective tensor norm on $\xtx$, then

\begin{theokwapien}
The multilinear operator $\Txxy$ factors through a Hilbert space if and only if there is a constant $C>0$ such that
\begin{equation*}
\sum\limits_{i=1}^{m}\|T(x_i^1,\dots, x_i^n)-T(z_i^1,\dots, z_i^n)\|^2  \leq  C^2 \sum\limits_{i=1}^{m}\pi(s_i^1\tens\dots\tens s_i^n-t_i^1\tens\dots\tens t_i^n)^2
\end{equation*}
holds for all finite sequences $(x_i^1,\dots, x_i^n)_{i=1}^m$, $(z_i^1, \dots, z_i^n)_{i=1}^m$, $(s_i^1, \dots, s_i^n)_{i=1}^m$ and $( t_i^1,\dots, t_i^n)_{i=1}^m$ in $\xpx$ with the property:
\begin{small}
\begin{equation*}
\sumim|\fhi(x_i^1,\dots, x_i^n)-\fhi(z_i^1,\dots, z_i^n)|^2  \leq \sumim |\fhi(s_i^1,\dots, s_i^n)-\fhi(t_i^1,\dots, t_i^n)|^2
\end{equation*}
\end{small}
for all $\fhi$ in $\Lxx$.
In this situation, $\Gamma(T)$ is the best constant $C$.
\end{theokwapien}

\

Note that in the case $n=1$,  Diagram \eqref{multilineardiagram} reduces to
\begin{equation*}
\begin{array}{c}
\xymatrix{
X\ar[dr]_{A}\ar[rr]^-T   &                              &Y\\
                                      &H\ar[ur]_B                           &
}
\end{array},
\end{equation*}
where all the involved operators are linear and $M=A(X)=H$.  In this case, we get an equivalent formulation of  the well known linear characterization (essentially due to Kwapien), namely, that a bounded linear operator $T:X\into Y$ factors through a Hilbert space if and only if there exists a constant $C>0$ such that
$$\sumim \|T(x_i)\|^2\leq C^2 \sumim \|a_i\|^2$$
holds for all finite sequences $(x_i)_{i=1}^m$ and $(a_i)_{i=1}^m$ in $X$ with the property:
$$\sumim |x^*(x_i)|^2\leq \sumim |x^*(a_i)|^2 \qquad\forall x^*\in X^*.$$

In this way, Theorem \ref{kwapien} extends the formulation of S. Kwapie\'n of linear operators factoring through a Hilbert space to the multilinear setting. The interested reader is deeply encouraged to see the early formulation of J. Lindenstrauss and A. Pelczynski of this property in \cite[Prop. 5.2]{LP} and the subsequent versions of S. Kwapie\'n \cite[Prop. 3.1]{kwapien72a} and \cite[Th . 2']{kwapien72}. Also see \cite[Th. 2.4]{pisier86} of G. Pisier for an accessible  proof and a good exposition of this class in the linear setting.

In relation with other multilinear properties, we prove that every Hilbert-Schmidt multilinear operator \cite[Definition 5.2]{matos03}, as well as  every Lipschitz $2$-summing multilinear operator \cite[Definition 3.1]{angulo17},  factors through a Hilbert space.

We also prove that the class of multilinear operators which satisfy diagram \eqref{multilineardiagram} enjoys ideal properties. To explain this, let $\mathcal{L}$ denote the class of all bounded multilinear operators. If we denote by $\Gamma$ the subclass of $\mathcal{L}$ that consists of all bounded multilinear operators that factor through a Hilbert space with the function $\Gamma(\cdot)$, then the pair $[\Gamma,\Gamma(\cdot)]$ is a maximal multi-ideal in the sense of Floret and Hunfeld \cite{floret_hunfeld01}. This affirmation is a consequence of Proposition \ref{idealbehavior} and Theorem \ref{maximal} which establishes the multi-ideal nature of $\Gamma$ and maximality, respectively.

In duality with the maximal multi-ideal nature of the class $\Gamma$, we exhibit a  finitely generated tensor norm $\gamma$ which satisfies that
$$(\xtxty,\gamma)^*=\Gxxyd$$
and
$$(\xtxtyd,\gamma)^*\cap\Lxxy=\Gxxy$$
hold isometrically. This results are presented as Theorem \ref{tensorialrepresentation} and Corollary \ref{tensorialrepresentation2}, respectively.

\

Following the ideas developed throughout the paper, we also  introduce the notion of polynomials that can be factored through a Hilbert space, see Definition \ref{polynomial}, and state a Kwapie\'n type characterization for polynomials (see Theorem \ref{kwapienpol}).

To obtain these results, we have  applied   the general approach introduced in \cite{fernandez-unzueta17a}.  This approach is, basically, to study a multilinear map $T$ by means of its associated $\Sigma$-operator $f_T$ (see Subsection \ref{subs: notation}). Posing the problem of factoring $T$ through a Hilbert space in the context of $\Sigma$-operators allowed us to use the geometrical richness of the tensor products of Banach spaces. Moreover,  since bounded $\Sigma$-operators are Lipschitz mappings {\cite[Theorem 3.2]{fernandez-unzueta17a}}, this  approach   enables us to relate, naturally, multilinear operators that factor through a Hilbert space with the metric  study carried  out  in  \cite{chavez14} (see Subsection \ref{subs: metric case}).

The material is organized as follows: In subsection 1.1 we fix some standard notation of Banach spaces and multilinear theories. We also recall from \cite{fernandez-unzueta17a} the notion of a $\so$. In section 2, we give the precise definition of a multilinear operator that  factors through a Hilbert space and present some examples. Section 3 is dedicated to prove the main result, that is, Theorem \ref{kwapien}. Section 4 is devoted to proving that the class $\Gamma$ of all multilinear operators that factor through a Hilbert space is a maximal multi-ideal. The duality with the tensor norm is proved  in  Theorem \ref{tensorialrepresentation} and Corollary \ref{tensorialrepresentation2}. In Section 5 we study the polynomials that factor through a Hilbert space, proving  a Kwapie\'n type characterization for them.

\subsection{Notation and Preliminaries}\label{subs: notation}

We use standard notation of the theory of Banach spaces. The letter $\kk$ denotes the real or complex numbers. The unit ball of a the normed space $X$ is denoted by $B_X$. We denote by $K_X:X\into X^{**}$ the canonical embedding.

Throughout this work, $n$ denotes a positive integer and the capital letters $\xxx,Y$ and $Z$ denote Banach spaces over the same field. The symbol $\Lxxy$ denotes the Banach space of bounded multilinear operators $\Txxy$ with the usual norm $\|T\|=\sup \{ \|T(x^1,\dots , x^n)\|  |  x^i\in B_{X_i}\}$. For simplicity of  notation, we write $\mathcal{L}(\xxx)$ in the case $Y=\kk$.

The set of decomposable tensors of the algebraic tensor product $\xtx$ is denoted by $\sxx$. This is,
$$\sxx:=\left\{\;  \xxt  \;|\;  x^i\in X_i  \;\right\}.$$
Let $\pi$ be the projective tensor norm given by
$$\pi(u;\xtx)= \inf \left\{\;  \sumim \|x_i^1\|\dots\|x_i^n\|  \;\Big|\; u=\sumim x_i^1\tens\dots\tens x_i^n \;\right\}.$$
The symbol $\sxxp$ denotes the resulting metric space obtained by restricting the norm $\pi$ to $\sxx$.

The universal property of the projective tensor product establishes that for every bounded multilinear operator $T:\xpx\into Y$ there exist a unique bounded linear operator $\tlin:\xtxp\into Y$ such that $T\xxp=\tlin(\xxt)$. In particular, the restriction of $\tlin$ to $\sxxp$ is a Lipschitz function. In this situation, the linear map $\tlin$ is called the linearization of $T$ and $f_T=\tlin|_{\sxxp}:\sxxp\into Y$ is called the $\so$ associated to $T$. Moreover, we have $\|T\|=Lip(f_T)=\|\tlin\|$ (for details on $\sos$ the reader may see  \cite{fernandez-unzueta17a}).

A norm $\beta$ on $\xtx$ is said to be a reasonable crossnorm if
$$\varepsilon(u)\leq\beta(u)\leq\pi(u)\qquad\forall\, u\in\xtx,$$
where $\varepsilon$ denote the injective tensor norm defined by
$$\varepsilon(u;\xtx)= \sup \left\{\;  |x_1^*\tens\dots\tens x_n^*(u)|  \;\Big|\; x_i^*\in B_{X_i^*},\; 1\leq i\leq n\;\right\}$$
for each $u$ in $\xtx$. The reader interested in tensor norms may check references \cite{defant93,ryan02,floret_hunfeld01}.

According to Theorem 2.1 of \cite{fernandez-unzueta17a} we have that if $\beta$ is a reasonable crossnorm on $\xtx$, then the resulting metric space $\sxxb$ (obtained by restricting the norm $\beta$ to $\sxx$) is Lipschitz equivalent to $\sxxp$. Specifically we have
\begin{equation}\label{metrics}
\pi(\pmq)\leq 2^{n-1}\, \beta(\pmq) \qquad\forall\,  p,q\in\sxx.
\end{equation}
for all reasonable crossnorm $\beta$ on $\xtx$.


\section{Definition, Examples and the Metric Case}

\begin{definition}
We say that the multilinear operator $T:\xpx\into Y$ factors through a Hilbert space if there exists a Hilbert space $H$, a subset $M$ of $H$, a bounded multilinear operator $A:\xpx\into H$ whose image is contained in $M$, and a Lipschitz function $B:M\into Y$ such that the diagram
\begin{equation}\label{factorization}
\begin{array}{c}
\xymatrix{
\xpx\ar[dr]_{A}\ar[rr]^-T   &                              &Y\\
                                              &M\ar[ur]_B\ar[d]  &\\
                                              &H                           &
}
\end{array}
\end{equation}
commutes. We define $\Gamma(T)$ as $\inf \|A\|\, Lip(B)$ where the infimum is taken over all possible factorizations as above.
\end{definition}

In the previous definition it is enough to take $M$ as $A(\xpx)$ (or equivalently, its closure in $H$).

The collection of multilinear operators $\Txxy$ which admit a factorization through a Hilbert space as in \eqref{factorization} is denoted by the symbol $\Gamma\left(\xxx; Y\right)$. The symbol $\Gamma$ denotes the class of all bounded multilinear operators that factors through a Hilbert space.

It is easy to see that the translation of the diagram \eqref{factorization} to the setting of $\sos$ acquires the form
\begin{equation}\label{sigmafactorization}
\begin{array}{c}
\xymatrix{
\sxxp\ar[dr]_{f_A}\ar[rr]^-{f_T}   &                              &Y\\
                                              &f_A(\sxx)\ar[ur]_B\ar[d]  &\\
                                              &H                           &
}
\end{array},
\end{equation}
where $f_T$ and $f_A$ are the $\sos$ associated to the bounded multilinear operators $T$ and $A$. In other words, $f_T=Bf_A$.

\subsection{Examples} In the sequel, $H_1\dots H_n$ and $H$ denote Hilbert spaces, $H_1\widehat{\tens}_2\dots\widehat{\tens}_2 H_n$ denotes its Hilbert tensor product and $\|\cdot\|_2$ denotes its reasonable crossnorm (we refer the reader to \cite[Sec. 2.6]{kadison83} for the details of this construction).

\subsection*{The Canonical Multilinear Map on Hilbert spaces}

The canonical multilinear map
\begin{eqnarray*}
\tens:H_1\times\dots\times H_n &\into &		  H_1\widehat{\tens}_\pi\dots\widehat{\tens}_\pi H_n\\
	\xxp					&\mapsto &  \xxt
\end{eqnarray*}
factors through the Hilbert space $H_1\widehat{\tens}_2\dots\widehat{\tens}_2 H_n$. Moreover \eqref{metrics} implies that
$$\Gamma(\tens)\leq \|\tens\|\, Lip(Id:\sig_{H_1\dots H_n}^{\|\cdot\|_2}\into \sig_{H_1\dots H_n}^{\pi})\leq 2^{n-1}.$$

Notice that the linearization $\tilde{\tens}$ coincides with the identity map on $H_1\widehat{\tens}_\pi\dots\widehat{\tens}_\pi H_n$. Hence, $\widehat{\tens}$ does not factor through a Hilbert space in the linear sense since for $n>1$, $H_1\widehat{\tens}_\pi\dots\widehat{\tens}_\pi H_n$   contains a subspace isometric to $\ell_1$ (see \cite[Ex. 2.10]{ryan02}).

As a consequence of the previous discussion, $T\in\Gxxy$ does not imply that its linearization $\tlin:\xtxp\into Y$ factors through a Hilbert space. However, the converse is naturally true.

\subsection*{Multilinear Operators Whose Linearization Factor through a Hilbert Space}

Consider a bounded multilinear operator $\Txxy$ whose linearization $\tlin:\xtxp\into Y$ factors through a Hilbert space. Notice that a typical factorization $\tlin=BA$ implies that $T$ factors as $T=B|_{f_A(\sxx)} (A\tens)$. Therefore, $T\in\Gxxy$ and $\Gamma(T)\leq \Gamma(\tlin)$. In other words, the operator
\begin{eqnarray*}
\Gamma(\xtxp;Y)  &\into &  \Gxxy\\
                   \tlin      &\mapsto &  T
\end{eqnarray*}
is bounded and has norm $\leq 1$.

\subsection*{When every Factor  of the Domain is a Hilbert Space}

It is natural to expect that operators of the form $T:H_1\times\dots\times H_n\into Y$ factor through the Hilbert space $H_1\widehat{\tens}_2\dots\widehat{\tens}_2 H_n$. Indeed, the identity
\begin{eqnarray*}
\Gamma(H_1, \dots , H_n;Y)  &\into &  \mathcal{L}(H_1, \dots, H_n; Y)\\
                   T      &\mapsto &  T
\end{eqnarray*}
is an onto isomorphism for every Banach space $Y$. To see this, let $T:H_1\times\dots\times H_n\into Y$ be a bounded multilinear operator. Recall that $f_T:\sig_{H_1\dots H_n}^{\pi}\into Y$ is a Lipschitz function and $\tens:H_1\times\dots\times H_n\into H_1\widehat{\tens}_2\dots\widehat{\tens}_2 H_n$ is bounded. From \eqref{metrics} we have that $f_T:\sig_{H_1\dots H_n}^{\|\cdot\|_2}\into Y$ is also Lipschitz. Hence, the factorization $T=f_T\tens$ tells us that $T\in\Gamma(H_1, \dots , H_n;Y)$ and
$$\Gamma(T)\leq \|\tens\|\, Lip(f_T:\sig_{H_1\dots H_n}^{\|\cdot\|_2}\into Y)\leq 2^{n-1}\|T\|.$$
Furthermore, (iii) of Proposition \ref{idealbehavior} says $\|T\|\leq \Gamma(T)$.

\subsection*{Hilbert-Schmidt Multilinear Operators}

Following Matos \cite{matos03}, let $\mathcal{L}_{HS}(H_1,\dots,H_n;H)$ denote the Banach space of Hilbert-Schmidt multilinear operators $T:H_1\times\dots\times H_n\into H$ endowed with the norm
$$\|T\|_{HS}=\left(\sum\limits_{ \substack{ j_i\in J_i \\  1\leq i\leq n} } \|T(e_{j_1}^1,\dots, e_{j_n}^n))\|^2\right)^{\frac{1}{2}},$$
where $(e_{j}^i)_{j\in J_i}$ is an orthonormal basis of $H_i$.

The previous example tells us that every Hilbert-Schmidt multilinear operator factors through a Hilbert space. Even more:

\begin{proposition}\label{HS} $\|Id:\mathcal{L}_{HS}(H_1,\dots,H_n;H)\into \Gamma(H_1,\dots,H_n;H) \|\leq 1.$
\end{proposition}
\begin{proof}
Recall that $\mathcal{L}_{HS}(H_1,\dots,H_n;H)$ is isometrically isomorphic to $\mathcal{L}_{HS}(H_1\widehat{\tens}_2\dots\widehat{\tens}_2 H_n;H)$ via the assignment $T\mapsto \tlin$, see \cite[Prop. 5.10]{matos03}. Then, every Hilbert-Schmidt multilinear operator $T:H_1\times\dots\times H_n\into H$ factors as $T=f_T\tens$ through $H_1\widehat{\tens}_2\dots\widehat{\tens}_2 H_n$. Moreover,
$$\|f_T(p)-f_T(q)\|\leq\|\pmq\|_2\,  \|T\|_{HS}\qquad\forall  p,q\in\sig_{H_1,\dots,H_n}.$$
Hence $\Gamma(T)\leq \|\tens\|\,  Lip(f_T:\sig_{H_1\dots H_n}^{\|\cdot\|_2}\into H  )\leq \|T\|_{HS}$.
\end{proof}

With this, we obtain   that for every $2\leq p < \infty$,  every fully absolutely
$p$-summing operator  $T\in \mathcal{L}_{fas}^p(H_1,\dots,H_n;H)$  (for this notion, see \cite[Def. 2.2]{matos03}) factors  through a Hilbert space. In Proposition 5.5 of this  reference,  the author proves that $\mathcal{L}_{fas}^2(H_1,\dots,H_n;H)$ is isometrically isomorphic to $\mathcal{L}_{HS}(H_1,\dots,H_n;H)$. Hence, the  previous proposition asserts that every absolutely $2$-summing multilinear operator $T$ between Hilbert spaces factors through a Hilbert space and $\Gamma(T)\leq \|T\|_{fas,2} $. Even more, according to \cite[Prop. 5.7]{matos03}, $\mathcal{L}_{fas}^p(H_1,\dots,H_n;H)$ is isomorphic to $\mathcal{L}_{HS}(H_1,\dots,H_n;H)$ for $2\leq p<\infty$. As a consequence, every fully absolutely $p$-summing  multilinear operator $T$ factors through a Hilbert space and $\Gamma(T)\leq (b_p)^n \|T\|_{fas,p}$, where $b_p$ is the greater constant from Khintchine's inequality for all $2\leq p<\infty$.

Notice that the morphism in Proposition \ref{HS} is not onto since $\tens:H_1\times\dots\times H_n \into H_1\widehat{\tens}_2\dots\widehat{\tens}_2 H_n$ is a bounded multilinear operator which is not Hilbert-Schmidt when each $H_i$ is infinite dimensional. The same observation is also valid for the case of fully absolutely summing multilinear operators we deal with before.

\subsection*{Lipschitz $2$-Summing Multilinear Operators}

In this example we relate the notion of Lipschitz $2$-summing multilinear operators developed in \cite{angulo17} with multilinear operators that factor through a Hilbert space.

One of the equivalences of the Lipschitz $2$-summability of $\Txxy$ (see, \cite[Th. 1.1]{angulo17}) establishes that $T$ factors as
\begin{equation*}
\begin{array}{c}
\xymatrix{
\xpx\ar[d]_{i}\ar[r]^-T                                &          Y                    \\
i(\sxx)\ar[d]\ar[r]_{j_2|_{i(\sxx)}}             &j_2\circ i (\sxx)\ar[d]\ar[u]_u \\
C(B_{\Lxx^*})\ar[r]_{j_2}                           & L_2(\mu)
}
\end{array},
\end{equation*}
where $\mu$ is a probability measure on $(B_{\Lxx^*},w^*)$, $i:\xpx\into C(B_{\Lxx^*})$ acts by evaluation, $j_2:C(B_{\Lxx^*})\into L_2(\mu) $ is the canonical inclusion and $u$ is a Lipschitz function such that $\pi_2^{Lip}(T)= Lip(u)$. Hence $T\in\Gxxy$ and $\Gamma(T)\leq\pi_2^{Lip}(T)$.

Actually, we have proved that
$$\|Id:\Pi_2^{Lip}(\xxx;Y)\into\Gxxy\|\leq 1$$
holds for all Banach spaces $\xxx$ and $Y$.

\subsection{Relation with the Metric Case}\label{subs: metric case}

Now, we turn our attention to  Lipschitz mappings between metric spaces. Recall from \cite{chavez14} that a Lipschitz function between metric spaces $f:X\into Y$ factors through a subset of a Hilbert space if there exist a Hilbert space $H$ and a subset $Z$ of $H$ (actually we may take $Z=f(X)$) and two Lipschitz functions $A:X\into Z$, $B:Z\into Y$ such that $f=BA$. In this case $\gamma_2^{Lip}(f)=\inf Lip(A)\,  Lip(B)$, where the infimum is taken over all possible factorizations of $f$ as before.

It is clear from Diagram \eqref{sigmafactorization} that if $T$ is an element in $\Gxxy$, then its associated $\so$ $f_T:\sxxp\into Y$ is a Lipschitz function that can be factored through a subset of a Hilbert space in the sense of \cite{chavez14}. Moreover, we have
$$\gamma_2^{Lip}(f_T)\leq \Gamma(T).$$
In other words, every multilinear operator $T$ in $\Gxxy$ gives rise to a Lipschitz function $f_T$ in $\Gamma_2^{Lip}(\sxxp; Y)$. That is, the operator
\begin{eqnarray}\label{metriccomparison}
\Gxxy &\into & \Gamma_2^{Lip}(\sxxp; Y)  \\
      T  &\mapsto &  f_T\nonumber
\end{eqnarray}
is bounded and has norm $\leq 1$.

We do not know if the metric approach of \cite{chavez14} restrict well to the setting of multilinear operators we are proposing. Specifically, we have two questions:

\textbf{Question 1:} Is the map defined in \eqref{metriccomparison} an isometry?

\textbf{Question 2:} We do not know if $T$ factors through a Hilbert space whenever $f_T$ does in the metric sense, that is, Does $f_T$ in $\Gamma_2^{Lip}(X;Y)$ imply $T$ in $\Gxxy$?


\section{Kwapie\'n Type Characterization}

In this section we  characterize the multilinear operators that factor through a Hilbert space,  in terms of their behavior on some special  finite sequences (Theorem \ref{kwapien}). This fact relies on the local character of the property of factoring through  a Hilbert space, which is proved in  Theorem \ref{maximal}.   First, we precise some needed facts and notation.

Sets of the form $f_A(\sxx)=A(\xpx)$, where $A:\xpx\into Z$ is bounded, are fundamental for the proof of Theorem \ref{maximal}. We collect some relevant facts about these sets in the next lemma. We omit its proof since it can be done by standard arguments of the theory of Banach spaces.
\begin{lemma}\label{sigmaimages}
Let $A:\xpx\into Z$ be a bounded multilinear operator between Banach spaces. Then:
\begin{itemize}
\item [i)] The set $(f_A(\sxx))^*$ defined by
$$\left\{ \psi:f_A(\sxx)\into \kk \;\Big|\; \psi A \mbox{ is multilinear and } \psi\mbox{ is Lipschitz} \right\}$$
is a vector space endowed with the algebraic operations defined pointwise; moreover, it becomes a Banach space with the Lipschitz norm induced by $Z$.
\item [ii)] Let $B:f_A(\sxx)\into Y$ be a Lipschitz function with $BA:\xpx\into Y$ multilinear. The function
\begin{eqnarray}\label{adjoint}
B^*:Y^* &\into& (f_A(\sxx))^*\\
y^*&\mapsto& y^*B.\nonumber
\end{eqnarray}
is a well defined bounded linear operator and $\|B^*\|\leq Lip(B)$. The linear operator $B^*$ is called the adjoint of $B$.
\end{itemize}
\end{lemma}

Let $E_i$ be a finite dimensional subspace of $X_i$, for $1\leq i\leq n$, and $L$ be a finite codimensional subspace of $Y$. We define the multilinear map
\begin{eqnarray*}
I_{\eee}:\epe  &\into &          \xtxp\\
             \xxp  &\mapsto &  \xxt
\end{eqnarray*}
and denote by $Q_L:Y\into Y/L$ the natural quotient map.

\begin{theorem}\label{maximal}
The multilinear operator $T:\xpx\into Y$ admits a factorization through a Hilbert space if and only if
$$s:=\sup \Gamma(Q_L f_T I_{\eee}) <\infty,$$
where the supremum is taken over all finite dimensional subspaces $E_i$ of $X_i$ and finite codimensional subspaces $L$ of $Y$. In this situation $\Gamma(T)=s$.
\end{theorem}

\begin{proof}
Suppose that $T:\xpx\into Y$ factors through a Hilbert space. Let $E_i$ and $L$ be as above. The factorization $T=BA$ implies $Q_L f_T I_{\eee}=(Q_LB)(f_A I_{\eee})$ and
$$\Gamma(Q_L f_T I_{\eee})\leq \Gamma(T).$$
Therefore, $s$ must be finite.

For the converse, we have to translate a condition on finite dimensional spaces to a global condition. For this end we use the technique of ultraproducts. Basic facts about ultraproducts of Banach spaces can be found in \cite{heinrich80}.

Let us denote by $\mathcal{F}(X)$ the collection of all finite dimensional subspaces of $X$ and by $\mathcal{CF}(Y)$ the collection of all finite codimensional subspaces of $Y$.

Define $\mathcal{P}=\mathcal{F}(X_1)\times\dots\times \mathcal{F}(X_n)\times\mathcal{CF}(Y)$. The relation $\leq$ given by $(\eel)\leq(\mmn)$ if $E_i\subset M_i$ and $N\subset L$ defines a partial order on $\mathcal{P}$. Let $\mathfrak{A}$ be an ultrafilter on $\mathcal{P}$ containing the sets
$$(\eel)^\#=\{\;  (\mmn)  \;|\; (\eel)\leq(\mmn) \;\}.$$

For each $(\eel)\in\mathcal{P}$ there exist a factorization as follows
$$\begin{array}{c}
\xymatrix{
\epe\ar[dr]_{{A^{\eee L}}}\ar[rr]^-{Q_L f_T I_{\eee}}  &                              &Y/L\\
                                              &A^{\eee L}\left(\epe\right)\ar[ur]_{B^{\eee L}}\ar[d]  &\\
                                              &H^{\eee L}                           &
}
\end{array},$$
with $\|A^{\eee, L}\| \leq 1$ and $Lip(B^{\eee, L})\leq s$. By the finite dimensional hypothesis, we may assume that $H^{\eee, L}=\ell_2^{n(\eel)}$, where $n(\eel)$ is a positive integer.

For each $(\eel)\in\mathcal{P}$ define
\begin{eqnarray*}
A_{\eee, L}:\xpx  &\into &       \ell_2^{n(\eel)}\\
                \xxp      &\mapsto &  \left\{\begin{array}{ccl}
                                                                    A^{\eee, L}\xxp & \mbox{if}\;\; \xxp\in\epe\\
                                                                            0            &  otherwise
                                                   \end{array}\right. .
\end{eqnarray*}
It is not difficult to see that
\begin{eqnarray*}
A:\xpx  &\into &\left(\ell_2^{n(\eel)}\right)_{\mathfrak{A}}\\
	\xxp  &\mapsto &  \left(A_{\eee, L}\xxp\right)_{\mathfrak{A}}
\end{eqnarray*}
is a multilinear mapping. Moreover,
\begin{eqnarray*}
\|A\xxp\|_{\mathfrak{A}} &=& \|(A_{\eee, L}\xxp)_\mathfrak{A}\|_{\mathfrak{A}} \\
                                &=&    \lim\limits_{\mathfrak{A}} \|A_{\eee, L}\xxp\| \\
                          &\leq  &    \|x^1\|\dots\|x^n\|
\end{eqnarray*}
implies that $A$ is a bounded and $\|A\|\leq 1$.

On the other hand, we extend the operator $(B^{\eee, L})^*:(Y/L)^*\into (f_{A^{\eee, L}}(\seee))^*$  (see Lemma \ref{sigmaimages}) as follows
\begin{eqnarray*}
\overline{(B^{\eee, L})^*} : Y^*  &\into & \left(f_{A^{\eee, L}}(\seee)\right)^*\\
                                y^*         &\mapsto &  \left\{\begin{array}{ccl}
                                               \left(B^{\eee, L}\right)^*(\zeta) &&  \mbox{if}\;\; y^*=Q_L^*(\zeta)\in Q_L^*((Y/L)^*)\\
                                                                            0  & &  otherwise
\end{array}
\right..
\end{eqnarray*}
Define
\begin{eqnarray*}
B:A(\xpx) &\into &       Y^{**}\\
      A\xxp     &\mapsto &  BA\xxp,
\end{eqnarray*}
where
\begin{eqnarray*}
BA\xxp:Y^*  &\into &\kk\\
y^*  &\mapsto &  \lim\limits_{\mathfrak{A}} \lev  \overline{(B^{\eee, L})^*}(y^*)  \,,\, A_{\eee, L}\xxp \rev.
\end{eqnarray*}

The definitions of $\overline{(B^{\eee, L})^*}$ and $A_{\eee, L}$ imply that
\begin{eqnarray}\label{filter}
\Big|\lev \overline{(B^{\eee, L})^*}(y^*) ,  A_{\eee, L}\xxp \rev -\lev \overline{(B^{\eee, L})^*}(y^*)  ,  A_{\eee, L}\zzp \rev \Big|& \nonumber\\
    \leq    s\,  \|y^*\|\,  \|A_{\eee, L}\xxp-A_{\eee, L}\zzp\|&
\end{eqnarray}
holds for all $y^*$ in $Y^*$, $\xxp$, $\zzp$ in $\xpx$ and $(\eel)$ in $\mathcal{P}$.

Inequality \eqref{filter} has many implications. First, $\zzt=0$ implies that $BA\xxp$ is well defined. Second, $(A^{\eee, L}\xxp)_{\mathfrak{A}}=(A^{\eee, L}\zzp)_{\mathfrak{A}}$ asserts that $B$ does not depend on representations since
$$\lim\limits_{\mathfrak{A}} \|A^{\eee, L}\xxp-A^{\eee, L}\zzp\|=0.$$
Third, the general case ensures that $B$ is Lipschitz and $Lip(B)\leq s$.

To conclude, note that for every $\xxp\in\xpx$ and $y^*\in Y^*$ there exists $(\eel)$ in $\mathcal{P}$ such that $\xxp\in\epe$ and $y^*\in Q_L^*((Y/L)^*)$. Then $(\eel)^\#\in\mathfrak{A}$ ensures that
$$\lim\limits_{\mathfrak{A}} \lev  \overline{(B^{\eee, L})^*}(y^*)  \,,\, A_{\eee, L}\xxp \rev=y^*(T\xxp).$$
As a consequence, $BA\xxp=K_Y T\xxp$ for all $\xxp$ in $\xpx$. This means that
$$\begin{array}{c}
\xymatrix{
\xpx\ar[dr]_{A}\ar[rr]^-T  &                                            &Y\\
                                          &A\left(\xpx\right)\ar[ur]_-{K_Y^{-1}B}\ar[d]  &\\
                                          &\left(\ell_2^{n(\eel)}\right)_{\mathfrak{A}}                           &
}
\end{array}$$
is commutative. Now, if we consider all the spaces $\ell_2^{n(\eel)}$ as abstract $L_2$-spaces, then the ultraproduct $\left(\ell_2^{n(\eel)}\right)_{\mathfrak{A}}$ is an abstract $L_2$-space. Moreover, \cite[Vol. II Th. 1.b.2]{lindenstrauss_tzafriri96} implies that this ultraproduct is (order) linearly isometric to $L_2(\mu)$ for some measure space $(\Omega,\mu)$. This means that $T$ factors through a Hilbert space and $\Gamma(T)\leq s$.
\end{proof}

\

Given finite sequences $(p_i)_{i=1}^m$, $(q_i)_{i=1}^m$, $(a_j)_{j=1}^l$, $(b_j)_{j=1}^l$ in $\sxx$ we write $\left(p_i,q_i\right)\leq_\pi\left(a_j,b_j\right)$ if
\begin{equation*}
\begin{array}{cr}
\sum\limits_{i=1}^m |f_\fhi(p_i)-f_\fhi(q_i)|^2\leq\sum\limits_{j=1}^l|f_\fhi(a_j)-f_\fhi(b_j)|^2 &\hfill\forall\; \fhi\in\Lxx.
\end{array}
\end{equation*}
Notice that it is enough, by adding zeros if necessary, to take $m=l$.

\begin{theorem}\label{kwapien}
The multilinear operator  $T:\xpx\into Y$ admits a factorization through a Hilbert space if and only if there exists a constant $C>0$ such that
\begin{equation}\label{kwapien1}
\sum\limits_{i=1}^{m}\|T(x_i^1,\dots, x_i^n)-T(z_i^1,\dots, z_i^n)\|^2  \leq  C^2 \sum\limits_{i=1}^{m}\pi(s_i^1\tens\dots\tens s_i^n-t_i^1\tens\dots\tens t_i^n)^2
\end{equation}
for all finite sequences such that $\left(x_i^1\tens\dots\tens x_i^n, z_i^1\tens\dots\tens z_i^n\right)  \leq_\pi  \left(s_i^1\tens\dots\tens s_i^n, t_i^1\tens\dots\tens t_i^n\right)$.
In this case, $\Gamma(T)$ is the best possible constant $C$ as above.
\end{theorem}

\begin{proof}
First, let us suppose that $T:\xpx\into Y$ admits a factorizations through a Hilbert space $H$, $T=BA$.   If $(p_i,q_i)\leq_\pi(a_i,b_i)$, then it is clear that $(f_A(p_i),f_A(q_i))\leq(f_A(a_i),f_A(b_i))$ in $H$. Given an orthonormal
basis $(e_\alpha)_{\alpha\in I}$ of $H$, we have that $\|h\|^2=\sum\limits_\alpha|\langle h,e_\alpha\rangle|^2$ holds for all $h\in H$. Then
\begin{eqnarray*}
\sum\limits_{i=1}^m\sum\limits_{\alpha\in F}|\lev f_A(p_i)-f_A(q_i) \,,\,  e_{\alpha}  \rev|^2  &=&  \sum\limits_{\alpha\in F}\sum\limits_{i=1}^m|\lev  f_A(p_i)-f_A(q_i)  \,,\,  e_{\alpha} \rev|^2  \\
                                                         &\leq& \sum\limits_{\alpha\in F}\sum\limits_{i=1}^m|\lev  f_A(a_i)-f_A(b_i)  \,,\,  e_{\alpha}\rev|^2
\end{eqnarray*}
for all finite subsets $F$ of $I$. Therefore
\begin{equation}\label{kwapien6}
\sum\limits_{i=1}^m\|f_A(p_i)-f_A(q_i)\|^2\leq \sum\limits_{i=1}^m\|f_A(a_i)-f_A(b_i)\|^2.
\end{equation}
Finally, the combination of \eqref{kwapien6} and the Lipschitz conditions of $B$ and $f_A$ imply
$$\sum\limits_{i=1}^m\|f_T(p_i)-f_T(q_i)\|^2 \leq Lip(B)^2\,  \|A\|^2\,  \sum\limits_{i=1}^m\beta(a_i-b_i)^2.$$
Consequently, (\ref{kwapien1}) must be true and $\inf C\leq Lip(B)\,  \|A\|$. Hence, $\inf C\leq\Gamma(T)$.

Conversely, let us prove that whenever $T$ satisfies  \eqref{kwapien1},  then $T$ admits such a factorization. For this end, we will use Theorem~\ref{maximal}. Let $E_i$ be a finite dimensional subspace of $X_i$. Let us denote by $\pi|$ the restriction of the norm $\pi(\cdot;  \xtx)$ to $\ete$. Set
$$K:=\left\{\; \zeta\in(\ete, \pi|)^*  \;|\; \|\zeta\|=1 \;\right\}.$$
Since the spaces $E_i$ are finite dimensional, $K$ is compact. Define $S$ the subset
of $C(K)$ given by the functions of the form
$$\phi(\zeta)=\sumim|\zeta(p_i)-\zeta(q_i)|^2-\sumim|\zeta(a_i)-\zeta(b_i)|^2,$$
where $(a_i)$, $(b_i)$, $(p_i)$ and $(q_i)$ are finite sequences in $\seee$ such that
$$C^2\,  \sumim \pi|(a_i-b_i)^2 < \sumim \|f_T(p_i)-f_T(q_i)\|^2.$$
Every element $\phi$ in $S$ satisfy $\|\phi\|>0$ since there exist $\zeta$ in $K$ such that $\phi(\zeta)>0$. Moreover, $S$ is a convex cone disjoint of the negative open cone $C_{-}:=\left\{\; \phi \;|\; \sup \phi<0 \;\right\}$. An application of the Hahn-Banach theorem ensures the existence of a measure $\mu$ on $K$ which separates $C_{-}$ and $S$. It is possible to adjust
$\mu$ to be a positive measure such that
\begin{equation}\label{kwapien2}
0\leq\int\limits_{K}\phi(\zeta)\,  d\mu(\zeta)\qquad\forall\,  \phi\in S.
\end{equation}
Since $E_i$ is a finite dimensional space
$$D=\sup\left\{ \left(\int\limits_{K}|\zeta(a)-\zeta(b)|^2 d\mu(\zeta)\right)^\frac{1}{2} \,\Big|\, \pi|(a-b)\leq 1,\, a,b\in\seee  \right\}>0.$$
Thus, we may adjust $\mu$ such that $D=C$.

For every $a,b,p,q\in\seee$ such that $C\;  \pi|(a-b)\leq\|f_T(p)-f_T(q)\|$, \eqref{kwapien2} asserts that
\begin{equation}\label{kwapien3}
\int\limits_{K}|\zeta(a)-\zeta(b)|^2\,  d\mu( \zeta)  \leq\int\limits_{K}|\zeta(p)-\zeta(q)|^2\,  d\mu(\zeta).
\end{equation}
In particular, (\ref{kwapien3}) is also true for $p$ and $q$ in $\seee$ such that $C<\|f_T(p)-f_T(q)\|$ and $a$, $b$ in $\seee$ with $\pi|(a-b)<1$. As a consequence
$$C\leq\left(\int\limits_{K}|\zeta(p)-\zeta(q)|^2\,  d\mu(\zeta)\right)^{\frac{1}{2}}$$
for all $p,q$ in $\seee$ with $C\leq\|f_T(p)-f_T(q)\|$. Take $c=\|f_T(p)-f_T(q)\|$ and $\varepsilon>0$. The homogeneous property of $f_T$ asserts that
$$C<(C+\varepsilon)\frac{c}{c}=\left\|f_T\left(\frac{C+\varepsilon}{c}p\right)-f_T\left(\frac{C+\varepsilon}{c}q\right)\right\|.$$
Hence,
\begin{equation*}
\begin{array}{cr}
\frac{C}{C+\varepsilon}\,  \|f_T(p)-f_T(q)\|\leq\left(\int\limits_{K}|\zeta(p)-\zeta(q)|^2d\mu\right)^\frac{1}{2} & \hfill\forall\; \varepsilon>0.
\end{array}
\end{equation*}
This way,
\begin{equation}\label{kwapien4}
\|f_T(p)-f_T(q)\|\leq\left(\int\limits_{K}|\zeta(p)-\zeta(q)|^2\,  d\mu(\zeta)\right)^\frac{1}{2}\qquad\forall \,  p,q\in\seee.
\end{equation}
On the other hand, it is clear that
\begin{equation}\label{kwapien5}
\left(\int\limits_{K}|\zeta(a)-\zeta(b)|^2\,  d\mu(\zeta)\right)^\frac{1}{2}\leq C\,  \pi|(a-b)\qquad\forall \;  a,b\in\seee.
\end{equation}
Finally, we obtain a factorization as follows
$$\begin{array}{c}
\xymatrix{
\epe\ar[dr]_{A}\ar[rr]^{f_TI_{\eee}}   &                              &  Y\\
                                               &f_A(\seee)\ar[d]\ar[ur]_B  &     \\
                                               &L_2(\mu)   &
}
\end{array},$$
where
\begin{eqnarray*}
A:\epe    &\into &		  L_2(\mu)\\
       \xxp &\mapsto &  A\xxp:\zeta\mapsto \zeta(\xxt)
\end{eqnarray*}
and
\begin{eqnarray*}
B:f_A(\seee)    &\into &		  Y\\
       f_A(\xxt) &\mapsto &  T\xxp.
\end{eqnarray*}
The boundedness of $A$ is deduced from \eqref{kwapien5}; moreover, $\|A\|\leq C$. Inequality \eqref{kwapien4} asserts that $B$ is a well defined Lipschitz function and $Lip(B)\leq1$.

Let $L$ be a finite codimensional subspace of $Y$ and consider $Q_L f I_{\eee}=(Q_LB)A$. Since $\|Q_L\|\leq1$ we obtain that $Q_LfI_{\eee}$ admits a factorization through a Hilbert space and $\Gamma(Q_L f_T I_{\eee})\leq C$. Theorem~\ref{maximal} implies that $T$ belongs to $\Gxxy$ and $\Gamma(T)\leq \inf C$.
\end{proof}


\section{Ideal Behavior and Tensorial Representation}

The ideal features of $\Gamma$ are contained in the next proposition. We omit its proof since it follows easily from the definition,  using the characterization provided by Theorem \ref{kwapien}.

\begin{proposition}\label{idealbehavior}
Let $\xxx$ and $Y$ be Banach spaces. Then:
\begin{itemize}
	\item [i)]  $\Gamma$ is a norm on $\Gamma\left(\xxx; Y\right)$.
	\item [ii)] Every rank-one multilinear operator
	\begin{eqnarray*}
    \fhi\cdot y:\xpx &\into& Y\\
    \xxp &\mapsto& \fhi\xxp\,  y
\end{eqnarray*}
	with $\fhi\in\mathcal{L}(\xxx)$ and $y\in Y$ is an element of $\Gamma\left(\xxx; Y\right)$ and $\Gamma(\fhi\cdot y)\leq \|\fhi\|\,  \|y\|$.
	\item [iii)] $\|T\|\leq \Gamma(T)$ for all $T\in\Gamma\left(\xxx; Y\right)$.
	\item [iv)]  Let $Z_1,\dots,Z_m$, $W$ be Banach spaces. Let $R:Z_1\times\dots\times Z_m\into\xtxp$ be a bounded multilinear operator such that $f_R\left(\Sigma_{Z_1,\dots,Z_m}\right)\subset\sxx$ and let $S:Y\into W$ be a bounded linear operator. Then $Sf_T R:Z_1\times\dots\times Z_m\into W$ is an element of $\Gamma\left(Z_1,\dots, Z_m; W\right)$ whenever $T$ is in $\Gamma\left(\xxx; Y\right)$ and $\Gamma(Sf_TR)\leq \|R\|\,  \Gamma(T)\,  \|S\|$.
\end{itemize}
\end{proposition}

A consequence of Corollary \ref{tensorialrepresentation2} is that every $\Gxxy$ is a Banach space. This result in addition to previous proposition and Theorem \ref{maximal} tells us that the pair $[\Gamma, \Gamma(\cdot)]$ is a maximal multi-ideal in the sense of K. Floret and S. Hunfeld \cite{floret_hunfeld01}.

In  \cite{floret_hunfeld01} the authors prove that every maximal ideal is represented by a finitely generated tensor norm, extending, in this way, the Representation Theorem for Maximal Ideals (see \cite[Sec. 17]{defant93}). Consequently, $\Gamma$ is represented by a  finitely generated tensor norm $\gamma$ with which it is  in duality. Now, we give an explicit fomulation of $\gamma$.

For a better understanding, it is convenient to have in mind that the mapping $\xxt\tens y\mapsto (\xxt)\tens y$ defines a linear isomorphism between $\xtxty$ and $\left(\xtx\right)\tens Y$. For example, under this identification if $p=\xxt$ and $q=\zzt$ are elements in $\sxx$ and $y$ is in $Y$, then
$$\xxt\tens y-\zzt\tens y=\pmqy.$$

In order to define $\gamma(u)$, the Lipschitz condition of bounded $\sos$ leads us to consider representations of $u$ of the form
\begin{equation}\label{angulorepresentation}
\sum\limits_{i=1}^m(\pimqi)\tens y_i,
\end{equation}
where $p_i$ and $q_i$ are elements in $\sxx$ and $y_i$ in $Y$.  The first tensor norm which considers representations as in \eqref{angulorepresentation} was given by Angulo in his doctoral dissertation \cite{angulo10}. In that case, J.C. Angulo defined the tensor norm $d_p$ which is in duality with the collection of Lipschitz $p$-summing multilinear operators defined in \cite{angulo17}.

Before presenting the norm $\gamma$, we fix some notation. Given finite sequences $(a_j)_{j=1}^m$ and $(b_j)_{j=1}^m$ in $\sxx$. We write
$$\|(a_j-b_j)\|_2^\pi:=\left(\sum\limits_{j=1}^m \pi(\ajmbj)^2\right)^{\frac{1}{2}}.$$
We also use the standard notation
$$\|(y_i)\|_2=\left(\sum\limits_{i=1}^m\|y_i\|^2\right)^{\frac{1}{2}}$$
for a finite sequence $(y_i)_{i=1}^m$ in Y.

\begin{definition}
Let $\xxx, Y$ be Banach spaces. For $u$ in $\xtxty$ define
$$\gamma(u)=\inf  \|(a_i-b_i)\|_2^\pi\,\|(y_i)\|_2 $$
where the infimum is taken over all representations $u=\sum\limits_{i=1}^m(\pimqi)\tens y_i$ and $(p_i,q_i)\leq_\pi (a_i,b_i)$.
\end{definition}

The  next proposition is straightforward and only requires standard techniques of tensor products. We omit the proof.

\begin{proposition}\label{gbtensornorm}
Let $\xxx, Y$ be Banach spaces and $\beta$ be a reasonable cross norm on the tensor product $\xtx$. Then:
\begin{itemize}
	\item [i)] $\gamma$ is a norm on $\xtxty$.	
	\item [ii)]   $\gamma(\pmqy)\leq\pi(\pmq)\,  \|y\|$ for all $p,q\in\sxx$ and $y\in Y$.
	\item [iii)]  Let $\fhi\in\Lxx$ and $y^*\in Y^*$. The functional
	\begin{eqnarray*}
\fhi\tens y^*:\left(\xtxty,\gamma\right) &\into &		  \kk\\
\xxt\tens y &\mapsto &  f_\fhi(\xxt)\,  y^*(y)
\end{eqnarray*}
is bounded and $\|\fhi\tens y^*\|\leq \|\fhi\|\,  \|y^*\|$.
	\item [iv)] Let $Z_1,\dots, Z_m, W$ be Banach spaces. If $R:Z_1\times\dots\times Z_m\into\xtxp$ is a bounded multilinear operator such that $f_R\left(\Sigma_{Z_1,\dots, Z_m}\right)\subset\sxx$ and  $S:W\into Y$ is a bounded linear operator then
	\begin{eqnarray*}
R\tens S:\left(Z_1\tens\dots\tens Z_m\tens W, \gamma\right) &\into &		  \left(\xtxty,\gamma\right)\\
z_1\tens\dots\tens z_m\tens w &\mapsto &  f_R(z_1\tens\dots\tens z_m)\tens S(w)
\end{eqnarray*}
is bounded and $\|R\tens S\|\leq \|R\|\,  \|S\|$.
	\item [v)] $\gamma(u;\, \xtxty)=\inf   \gamma(u;\,\etetf)$\\
	
	where the infimum is taken over all finite dimensional subspaces $E_i$ and $F$ of $X_i$ and $Y$, respectively, such that $u\in\etetf$.
\end{itemize}
\end{proposition}

The previous proposition tells us that $\gamma$ is a finitely generated tensor norm in the sense of Floret and Hunfeld \cite{floret_hunfeld01}.

For a well understanding of the tensorial representation of the class $\Gamma$, we precise the involved algebraic morphism. Every bounded multilinear operator $T:\xpx\into Y^*$ gives rise to a bounded functional
\begin{eqnarray*}
\fhi_T:\xtxtyp &\into & \kk\\
\xxt\tens y &\mapsto &  T\xxp(y).
\end{eqnarray*}
Conversely, every bounded functional $\fhi$ on $\xtxtyp$ defines a bounded multilinear operator
\begin{eqnarray*}
T_\fhi:\xpx &\into & Y^*\\
\xxt &\mapsto &  T_\fhi(\xxt):y\mapsto \fhi(\xxt\tens y).
\end{eqnarray*}
It is not difficult to prove that these assignments are linear isometries and inverse of each other. In other words, we have that
\begin{eqnarray}\label{boundedcase}
\Phi:(\xtxtyp)^*  &\into &       \mathcal{L}\left(\xxx; Y^*\right) \\
               \fhi               &\mapsto &  T_\fhi\nonumber
\end{eqnarray}
is an isometric linear isomorphism. The next theorem establishes that $\Phi$ in \eqref{boundedcase} also is an isometric linear isomorphism if we replace $\pi$ by the norm $\gamma$ and $\mathcal{L}\left(\xxx; Y^*\right)$ by the normed space $\Gxxyd$.

\begin{theorem}\label{tensorialrepresentation}
Let $\xxx, Y$ be Banach spaces. Then
$$\Phi:\left(\xtxty,\gamma\right)^*\into\Gxxyd$$
is an isometric linear isomorphism.
\end{theorem}

\begin{proof}
We will use the linear isometry
\begin{equation}\label{sum2}
\left(Y \oplus_2\dots\oplus_2 Y\right)^*=Y^*\oplus_2\dots\oplus_2 Y^*.
\end{equation}

Suppose that $T$ factors through a Hilbert space. The combination of \eqref{sum2} and Theorem~\ref{kwapien}, implies that for all $(y_i)_i$ and $(p_i,q_i)\leq_\pi(a_i,b_i)$,
$$\left|\sum\limits_{i=1}^m \lev f_T(p_i)-f_T(q_i) \,,\,  y_i \rev\right|\leq \Gamma(T)\,  \|(\aimbi)\|_2^\pi\,  \|(y_i)\|_2.$$
So, if $u=\rju$ is an element in $\xtxty$ and $(p_i,q_i)\leq_\pi(a_i,b_i)$ then
$$|\fhi_T(u)|\leq \Gamma(T)\,  \|(a_j-b_j)\|_2^\pi\,  \|(y_i)\|_2.$$
In other words $\fhi_T$ is bounded and $\|\fhi_T\|\leq\Gamma(T)$.

Conversely, suppose $\fhi\in \left(\xtxty,\gamma\right)^*$. Let $(p_i,q_i)\leq_\pi(a_i,b_i)$ and $(y_i)_i$. Define $u=\rju$, then
$$\left|\sum\limits_{i=1}^m\lev f_{T_\fhi}(p_i)-f_{T_\fhi}(q_i), y_i\rev\right|=|\fhi(u)|\leq \|\fhi\|\;  \|(\aimbi)\|_2^\pi\;  \|(y_i)\|_2.$$
After taking suprema over $\sumim \|y_i\|^2\leq 1$, \eqref{sum2} implies
$$\left(\sum\limits_{i=1}^m\|f_{T_\fhi}(p_i)-f_{T_\fhi}(q_i)\|^2\right)^{\frac{1}{2}}\leq \|\fhi\|\;  \|(\aimbi)\|_2^\pi.$$
According to Theorem \ref{kwapien}, $T_\fhi:\xpx\into Y^*$ factors through a Hilbert space and $\Gamma(T_\fhi)\leq \|\fhi\|$.
\end{proof}

\begin{corollary}\label{tensorialrepresentation2}
Let $\xxx$ and $Y$ be Banach spaces. Then, there exist an isometric isomorphism between $\Gxxy$ and $\left(\xtxtyd,\gamma\right)^*\bigcap\Lxxy$.
\end{corollary}

\begin{proof}
Let $T$ be an operator in $\Gxxy$. Define
\begin{eqnarray*}
\zeta_T :\xtxtyd &\into &		  \kk\\
\xxt\tens y^* &\mapsto &  y^*(T\xxp).
\end{eqnarray*}
The multilinear feature of $T$ and the linearity of every $y^*$ assert that $\zeta_T$ is well defined and linear. Let $u\in\xtxtyd$ and $\eta>0$. The finitely generated property of $\gamma$ (see Proposition \ref{gbtensornorm}) asserts that there exist $E_i\in\mathcal{F}(X_i)$ and $F\in\mathcal{F}(Y^*)$ such that $u\in\etetf$ and
$$\gamma(u;\eee\tens F)\leq(1+\eta)\,  \gamma(u;\xtxtyd).$$
The subspace $F$ defines $L\in\mathcal{CF}(Y)$ such that $(Y/L)^*=F$ holds isometrically isomorphic via $Q_L^*$. Then, Theorem \ref{tensorialrepresentation} implies that
\begin{equation}\label{link}
((\eee\tens (Y/L)^*), \gamma)^*= \Gamma(\eee;Y/L)
\end{equation}
holds isometrically isomorphic. Notice that in \eqref{link} we are identifying $Y/L$ with its double topological dual. Algebraic manipulations leads us that, under \eqref{link}, $Q_L f_T I_{\eee}$ is the multilinear operators that corresponds to the composition $\fhi_T\circ  (I_{\eee}\tens Q_L^*)$. Furthermore,
\begin{eqnarray*}
|\zeta_T(u)| &= &   |\fhi_T\circ  (I_{\eee}\tens Q_L^*) (u)|\\
                   &\leq& \|\fhi_T\circ  (I_{\eee}\tens Q_L^*):(\eee\tens (Y/L)^*), \gamma)\into \kk\|\,  \gamma(u)\\
                  &\leq & \Gamma(Q_L f_T I_{\eee})\,  (1+\eta)\,  \gamma(u;\xtxtyd)\\
                  &\leq & \Gamma(T)\, (1+\eta)\, \gamma(u;\xtxtyd).
\end{eqnarray*}
The election of $\eta$ allows us to conclude that $\zeta_T$ is bounded and $\|\zeta_T\|\leq \Gamma(T)$.

For the converse, let $\fhi\in\left(\xtxtyd,\gamma\right)^*\bigcap\Lxxy$. We may assume that $T_\fhi$ has range contained in $Y$. Reasoning as before (see \eqref{link}) we have that
$$\Gamma(Q_L f_{T_\fhi} I_{\eee}) =\|\fhi\circ (I_{\eee}\tens Q_L^*)\|\leq\|\fhi\|$$
holds for all $E_i\in\mathcal{F}(X_i)$ and $L\in\mathcal{CF}(Y)$. Hence, Theorem \ref{maximal} asserts that $T\in\Gxxy$ and $\Gamma(T)\leq\|\fhi\|$.

Finally, it is easy to check that the assignments $T\mapsto \zeta_T$ and $\fhi\mapsto T_\fhi$ are linear and inverse of each other.
\end{proof}

\subsection{Preservation of the property of factoring through a Hilbert space}

\subsubsection*{Decreasing the  Degree by Evaluations}

For any bounded multilinear operator $\Txxy$ and any $x^n$ in $X_n$ fixed define
\begin{eqnarray*}
T^{x^n}:X_1\times \dots\times X_{n-1} &\into & Y\\
(x^1,\dots, x^{n-1})&\mapsto &  T\xxp.
\end{eqnarray*}
Plainly, $T^{x^n}$ is a bounded multilinear operator. Analogously, we can define a bounded multilinear operator $T^{x^{n-k+1},\dots,x^n}:X_1\times \dots\times X_{n-k}\into Y$ for $1\leq k< n$ once we fix $x^j$ in $X_j$ for $n-k+1\leq j \leq n$.

\begin{proposition}\label{fixk}
Let $T$ in $\Gxxy$ and $x^j$ in $X_j$ for $n-k+1\leq j\leq n$. Then $T^{x^{n-k+1},\dots,x^n}$ is an element of $\Gamma(X_1,\dots X_{n-k}; Y)$ for all $1\leq k<n$ and $\Gamma(T^{x^{n-k+1},\dots,x^n})\leq \Gamma(T)\|x^{n-k+1}\|\dots\|x^{n}\|.$
\end{proposition}

\begin{proof}
It is enough to apply (iv) from Proposition \ref{idealbehavior} to the map $R:X_1\times \dots\times X_{n-k}\into \xtxp$ defined by $R(x^1,\dots, x^{n-k})=\xxp$.
\end{proof}

The case $k=n-1$ produces a bounded linear operator $T_{n-(n-1)}:X_1\into Y$ that factors through a Hilbert space in the linear sense. In this respect, we can say more, let $p=\xxp$ in $\xpx$ fixed and let $T$ in $\Gxxy$. Denote by $T_i:X_i\into Y$ the linear map resulting by fixing all coordinates except the $i$-th (see Proposition \ref{fixk}). Analogous arguments to those done in the proof of Proposition \ref{fixk} allows to conclude that $T_i$ is an element of $\Gamma(X_i;Y)$ and $\Gamma(T_i)\leq \Gamma(T)\prod\limits_{j\neq i}\|x^j\|$ for all $1\leq i\leq n$. Hence
$$\Gamma(T_1)\dots\Gamma(T_n)\leq \Gamma(T)^n \pi(p)^{n-1}.$$

\subsubsection*{Increasing the Degree by Products}

In the following proposition we show how to construct multilinear operators that factors through a Hilbert space for given $n$ linear operators with the same property.

\begin{proposition}
Let $n$ be a positive integer and let $T_i:X_i\into Y_i$ in $\Gamma(X_i; Y_i)$ for $1\leq i\leq n$. Then
\begin{eqnarray*}
\tens \circ (T_1,\dots, T_n):\xpx &\into &Y_1\widehat{\tens}_\pi\dots\widehat{\tens}_\pi Y_n\\
\xxp          &\mapsto & T_1(x^1)\tens\dots\tens T_n(x^n)\\
\end{eqnarray*}
belongs to $\Gamma(\xxx;Y_1\widehat{\tens}_\pi\dots\widehat{\tens}_\pi Y_n)$ and
$$\Gamma(\tens \circ (T_1,\dots, T_n))\leq 2^{n-1}\Gamma(T_1)\dots\Gamma(T_n).$$
\end{proposition}

\begin{proof}
Let $T_i=B_iA_i$ be a factorization through the Hilbert space $H_i$ for $1\leq i\leq n$. Since $A_i$  and $B_i$ are bounded for all $i$ we have that $A_1\tens\dots\tens A_n:\xtxp\into H_1\widehat{\tens}_\pi\dots\widehat{\tens}_\pi H_n$ and $B_1\tens\dots\tens B_n:H_1\widehat{\tens}_\pi\dots\widehat{\tens}_\pi H_n\into Y_1\widehat{\tens}_\pi\dots\widehat{\tens}_\pi Y_n$ are bounded. Applying (iv) of Proposition \ref{idealbehavior} to $R:\xpx\into H_1\widehat{\tens}_\pi\dots\widehat{\tens}_\pi H_n$, defined by $R\xxp=A_1(x^1)\tens\dots\tens A_n(x^n)$, $T=\tens:H_1\times\dots\times H_n\into H_1\widehat{\tens}_\pi\dots\widehat{\tens}_\pi H_n $ and $S=B_1\tens\dots\tens B_n$ we have that $(B_1\tens\dots\tens B_n) f_\tens R: \xpx\into Y_1\widehat{\tens}_\pi\dots\widehat{\tens}_\pi Y_n$ is in $\Gamma(\xxx;Y_1\widehat{\tens}_\pi\dots\widehat{\tens}_\pi Y_n)$ and $\Gamma((B_1\tens\dots\tens B_n) f_\tens R)\leq 2^{n-1}\prod\limits_{i=1}^n\|A_i\|\|B_i\|$. Hence $\Gamma((B_1\tens\dots\tens B_n) f_\tens R)\leq 2^{n-1}\Gamma(T_1)\dots\Gamma(T_n)$. We are done since $\tens \circ (T_1,\dots, T_n)=(B_1\tens\dots\tens B_n) f_\tens R$.
\end{proof}


\section{Polynomials Factoring through a Hilbert Space}
Homogenous polynomials that factorize through a Hilbert space can   also be  characterized    in terms of their behaviour in some special finite sequences of points.  In this case we only state the main results. Their   proofs are  analogous to those of  Theorems \ref{maximal} and \ref{kwapien}.

Recall that a  mapping $P:X\rightarrow Y$ between Banach spaces is a {\sl homogeneous poylnomial of degree $n$}
if there exists a multilinear mapping $T_P: X\times\ldots\times X \rightarrow Y$ such that $P(x)=T_P(x,\stackrel{n}{\ldots},x)$.

\begin{definition}\label{polynomial}
A $n$-homogeneous polynomial $P:X\into Y$ factors through a Hilbert space if there exist a Hilbert space $H$, a bounded $n$-homogeneous polynomial $q:X\into H$ and a Lipschitz function $B:q(X)\into Y$such that $p=Bq$. We define $\Gamma(q)=\inf \|q\|Lip(B)$.
\end{definition}

It is clear that every $T$ in $\Gamma(X\times\dots\times X\into Y)$ defines an n-homogeneous polynomial $p:X\into Y$ that factors through a Hilbert space and $\Gamma(p)\leq \Gamma(T)$. Also,  a composition of the form $SpR$ factors through a Hilbert space if $p$ does and $R$ and $S$ are bounded linear operators; moreover, $\Gamma(RpS)\leq \|R\|\Gamma(p)\|S\|$.

\begin{theorem}\label{maximalpol}
The $n$-homogeneous polynomial $p:X\into Y$ admits a factorization through a Hilbert space if and only if
$$s:=\sup \{\;  \Gamma(Q_L p I_{E}) \;|\; E\in\mathcal{F}(X), L\in\mathcal{CF}(Y) \;\} <\infty.$$
In this situation $\Gamma(p)=s$.
\end{theorem}

If we denote by $\pi_{n,s}$ the symmetric projective tensor norm on the symmetric tensor product $\tens^{n,s} X$ and $\tens^n x:=x\tens\dots\tens x$ we have:

\begin{theorem}\label{kwapienpol}
The $n$-homogeneous polynomial $p:X\into Y$ admits a factorization through a Hilbert space if and only if there exists a constant $C>0$ such that
\begin{equation*}
\sum\limits_{i=1}^{m}\|p(x_i)-p(z_i)\|^2  \leq  C^2 \sum\limits_{i=1}^{m}\pi_{n,s}(\tens^n s_i-\tens^n t_i)^2
\end{equation*}
for all finite sequences $(x_i)$, $(y_i)$, $(s_i)$ and $(t_i)$ in $X$ such that
$$\sumim |\fhi(x_i)-\fhi(z_i)|^2\leq \sumim |\fhi(s_i)-\fhi(t_i)|^2$$
for all $n$-homogeneous polynomial $\fhi:X\into \kk$. In this case, $\Gamma(p)$ is the best possible constant $C$ as above.
\end{theorem}

\bibliographystyle{amsplain}

\end{document}